\documentclass[11pt,reqno]{amsproc}
\usepackage{amsmath,amscd,amsthm,amssymb}
\usepackage{tikz-cd}

\newtheorem{theorem}{Theorem}[section]
\newtheorem{lemma}[theorem]{Lemma}
\newtheorem{proposition}[theorem]{Proposition}
\newtheorem{corollary}[theorem]{Corollary}

\theoremstyle{definition}

\newtheorem{definition}[theorem]{Definition}

\newtheorem{remark}[theorem]{Remark}

\newtheorem{example}[theorem]{Example}

\numberwithin{equation}{section}

\begin{document}

\title{The Lie groupoid analogue of a symplectic Lie group}

\author{David N. Pham}
\address{Department of Mathematics $\&$ Computer Science, QCC CUNY, Bayside, NY 11364}
\email{dpham90@gmail.com}
\thanks{This work was supported by PSC-CUNY Award $\#$ 60152-00 48.}



\date{}

\dedicatory{}

\begin{abstract}
A symplectic Lie group is a Lie group with a left-invariant symplectic form.  Its Lie algebra structure is that of a quasi-Frobenius Lie algebra.  In this note, we identify the groupoid analogue of a symplectic Lie group.  We call the aforementioned structure a \textit{$t$-symplectic Lie groupoid}; the ``$t$" is motivated by the fact that each target fiber of a $t$-symplectic Lie groupoid is a symplectic manifold.  For a Lie groupoid $\mathcal{G}\rightrightarrows M$, we show that there is a one-to-one correspondence between quasi-Frobenius Lie algebroid structures on $A\mathcal{G}$ (the associated Lie algebroid) and $t$-symplectic Lie groupoid structures on $\mathcal{G}\rightrightarrows M$.  In addition, we also introduce the notion of a \textit{symplectic Lie group bundle} (SLGB) which is a special case of both a $t$-symplectic Lie groupoid and a Lie group bundle.  The basic properties of SLGBs are explored. 
\end{abstract}

\maketitle
\section{Introduction}
A symplectic Lie group is a Lie group $G$ together with a left-invariant symplectic form $\omega$ \cite{Chu1974, BC2013}.  The associated Lie algebra structure is that of a quasi-Frobenius Lie algebra \cite{CP1994}; the latter is formally a Lie algebra $\frak{q}$ together with a skew-symmetric, non-degenerate bilinear form $\beta$ on $\frak{q}$ such that 
\begin{equation}
\nonumber
\beta([x,y],z)+\beta([y,z],x)+\beta([z,x],y)=0
\end{equation}
for all $x,y,z\in \frak{q}$.  In other words, $\beta$ is a non-degenerate 2-cocycle in the Lie algebra cohomology of $\frak{q}$ with values in $\mathbb{R}$ (where $\frak{q}$ acts trivially on $\mathbb{R}$).  For a symplectic Lie group $(G,\omega)$, the associated quasi-Frobenius Lie algebra is $(\frak{g},\omega_e)$, where $\frak{g}=T_eG$ is the Lie algebra defined by the left-invariant vector fields on $G$.

The notion of a \textit{quasi-Frobenius Lie algebroid} (or \textit{symplectic Lie algebroid} as it is more commonly called) was introduced independently in \cite{NT2001} and \cite{LMM2005}.  As one would expect, a quasi-Frobenius Lie algebroid over a point is simply a quasi-Frobenius Lie algebra.  As far as the author can tell, the Lie groupoid analogue of a symplectic Lie group has not been formally identified in the literature.  In other words, the following question has not yet been answered: \textit{what is the Lie groupoid structure whose assoicated Lie algebroid is precisely a quasi-Frobenius Lie algebroid?}  

To be clear, there is a structure in the literature called a \textit{symplectic Lie groupoid} \cite{Weinstein1987}.  However, it is unrelated to the notion of a symplectic Lie group.  Formally, a symplectic Lie groupoid is a Lie groupoid $\mathcal{G}\rightrightarrows M$ together with a symplectic form $\omega$ on $\mathcal{G}$ such that 
\begin{equation}
\nonumber
\mathcal{G}_3:=\{(g,h,gh)~|~(g,h)\in \mathcal{G}_2\}
\end{equation}
is a Lagrangian submanifold of $\mathcal{G}\times \mathcal{G}\times \overline{\mathcal{G}}$, where 
$\overline{\mathcal{G}}$ is the symplectic manifold $(\mathcal{G},-\omega)$ and
\begin{equation}
\nonumber
\mathcal{G}_2:=\{(g,h)~|~g,h\in \mathcal{G},~s(g)=t(h)\}.
\end{equation}
The condition that $\mathcal{G}_3$ is a Lagrangian submanifold of $\mathcal{G}\times \mathcal{G}\times \overline{\mathcal{G}}$ is equivalent to the condition that 
\begin{equation}
\label{SymplecticLieGroupoid1}
m^\ast\omega = \pi_1^\ast\omega+\pi_2^\ast\omega
\end{equation}
where $m: \mathcal{G}_2\longrightarrow \mathcal{G}$ denotes the multiplication map and $\pi_i: \mathcal{G}_2\longrightarrow \mathcal{G}$ denotes the natural projection map for $i=1,2$.   

Any Lie groupoid over a point is just a Lie group.  Hence, one might expect that a symplectic Lie groupoid over a point is just a symplectic Lie group, but this is not the case.  In fact, there are no symplectic Lie groupoids over a point.  To see this, let $\omega$ be a 2-form on a Lie group $G$ which satisfies (\ref{SymplecticLieGroupoid1}).  Let $g,h\in G$ and let
\begin{equation}
\nonumber
u:=(x,y),~v:=(x',y')\in T_g G\times T_h G.
\end{equation}
Then
\begin{equation}
\label{SymplecticLieGroupoid2}
(m^\ast\omega)_{(g,h)}(u,v)=\omega_{gh}((r_h)_{\ast} x+(l_g)_{\ast}y,(r_h)_{\ast} x'+(l_g)_{\ast}y')
\end{equation}
and
\begin{equation}
\label{SymplecticLieGroupoid3}
(\pi_1^\ast\omega)_{(g,h)}(u,v)+(\pi_2^\ast\omega)_{(g,h)}(u,v)=\omega_g(x,x')+\omega_h(y,y').
\end{equation}
Setting $h=e$, $x'=0_g$, and $y=0_e$ in (\ref{SymplecticLieGroupoid2}) and (\ref{SymplecticLieGroupoid3}) gives
\begin{equation}
\label{SymplecticLieGroupoid4}
(m^\ast\omega)_{(g,e)}(u,v)=\omega_{g}(x,(l_g)_{\ast}y')
\end{equation}
and 
\begin{equation}
\label{SymplecticLieGroupoid5}
(\pi_1^\ast\omega)_{(g,e)}(u,v)+(\pi_2^\ast\omega)_{(g,e)}(u,v)=0.
\end{equation}
Since $\omega$ satisfies (\ref{SymplecticLieGroupoid1}) by assumption, equations  (\ref{SymplecticLieGroupoid4}) and  (\ref{SymplecticLieGroupoid5})  imply that $\omega_g\equiv 0$ for all $g\in G$.  This shows that for any Lie group $G$, there are no symplectic forms which satisfy (\ref{SymplecticLieGroupoid1}).  Hence, there are no symplectic Lie groupoids over a point.  

In this note, we will identify the groupoid analogue of a symplectic Lie group.  We call the aforementioned structure a \textit{$t$-symplectic Lie groupoid}; the ``$t$" is motivated by the fact that each target fiber of a $t$-symplectic Lie groupoid is a symplectic manifold.  For a Lie groupoid $\mathcal{G}\rightrightarrows M$, we show that there is a one-to-one correspondence between quasi-Frobenius Lie algebroid structures on $A\mathcal{G}$ (the associated Lie algebroid) and $t$-symplectic Lie groupoid structures on $\mathcal{G}\rightrightarrows M$.   In addition, we also introduce the notion of a \textit{symplectic Lie group bundle} (SLGB) which is a special case of both a $t$-symplectic Lie groupoid and a Lie group bundle \cite{McK1, McK2}.       

The rest of this paper is organized as follows.  In section 2, we give a brief review of Lie groupoids and Lie algebroids.  In section 3, we introduce $t$-symplectic Lie groupoids, and establish the aforementioned one-to-one correspondence.  Some basic examples of $t$-symplectic Lie groupoids are also presented.   We conclude the paper in section 4 by introducing SLGBs and exploring some of its basic properties.  In addition, we also prove a result which is useful for the construction of nontrivial SLGBs.

\section{Preliminaries}
\subsection{Lie groupoids $\&$ Lie algebroids}
\noindent In this section, we give a brief review of Lie groupoids and Lie algebroids \cite{McK1, McK2, Marle, Marle1}, mainly to establish the notation for the rest of the paper.  We begin with the following definition:
\begin{definition}
\label{LieGroupoidDef1}
A \textit{Lie groupoid} is a groupoid $\mathcal{G}\rightrightarrows M$ such that
\begin{itemize}
\item[(i)] $\mathcal{G}$ and $M$ are smooth manifolds
\item[(ii)] all structure maps are smooth
\item[(iii)] the source map $s:\mathcal{G}\longrightarrow M$ is a surjective submersion.
\end{itemize}
\end{definition}
\begin{remark}
Note that condition (iii) of Definition \ref{LieGroupoidDef1} is equivalent to the condition that the target map $t:\mathcal{G}\longrightarrow M$ is a surjective submersion.  

In addition, the axioms of a Lie groupoid imply that the unit map 
\begin{equation}
\nonumber
u: M\longrightarrow \mathcal{G}
\end{equation}
is a smooth embedding.  As a consequence of this, we will often view $M$ as an embedded submanifold of $\mathcal{G}$.  With this viewpoint, $u$ is simply the inclusion map. 

The domain of the multiplication map $m$ on $\mathcal{G}$ is typically denoted as
\begin{equation}
\nonumber
\mathcal{G}_2:=\{(g,h)~|~g,h\in \mathcal{G},~s(g)=t(h)\}.
\end{equation}
Give $(g,h)\in \mathcal{G}_2$, we set $gh:=m(g,h)$.
\end{remark}

\begin{definition}
\label{LieGrpdMorph}
Let $\mathcal{G}\rightrightarrows M$ and $\mathcal{H}\rightrightarrows N$ be Lie groupoids.  Let  $(s,t)$ and ($s',t'$) denote the source and target maps of  $\mathcal{G}\rightrightarrows M$ and $\mathcal{H}\rightrightarrows N$ respectively.  Also, let $u$ and $u'$ denote the respective unit maps.   A homomorphism from $\mathcal{G}\rightrightarrows M$ to $\mathcal{H}\rightrightarrows N$ is a pair of smooth maps $F: \mathcal{G}\rightarrow \mathcal{H}$ and $f:M\rightarrow N$ such that 
\begin{itemize}
\item[(i)] $F(gh)=F(g)F(h)$ for all $(g,h)\in \mathcal{G}_2$
\item[(ii)] $F\circ u = u'\circ f$
\item[(iii)] $s'\circ F=f\circ s$, $t'\circ F=f\circ t$
\end{itemize}
\end{definition}

\begin{example}
Any Lie group is naturally a Lie groupoid over a point.
\end{example}

\begin{example}
Associated to any manifold $M$ is the \textit{pair groupoid} $M\times M\rightrightarrows M$ whose structure maps are defined as follows:
\begin{align}
\nonumber
s(p,q)&:=q,\hspace*{0.1in} t(p,q):=p,\hspace*{0.1in} (p,q)(q,r):=(p,r)\\
\nonumber
&u(p):=(p,p),\hspace*{0.1in} i(p,q):=(q,p)
\end{align}
for $p,q,r\in M$.
\end{example}

\begin{example}
Let $M$ be a manifold with a smooth left-action by a Lie group $G$.  Associated to $(M,G)$ is the \textit{action groupoid} $G\times M\rightrightarrows M$ whose structure maps are defined as follows:
\begin{align}
\nonumber
s(g,p):=g^{-1}p&,\hspace*{0.1in} t(g,p):=p,\hspace*{0.1in} (g,p)(h,g^{-1}p):=(gh,p)\\
\nonumber
u(p):&=(e,p),\hspace*{0.1in} i(g,p):=(g^{-1},g^{-1}p).
\end{align}
for $g,h\in G$, $p\in M$.
\end{example}

\begin{definition}
\label{LieAlgebroidDef}
A \textit{Lie algebroid} is a triple $(A,\rho,M)$ where $A$ is a vector bundle over $M$ and $\rho: A\rightarrow TM$ is a vector bundle map called the \textit{anchor} such that 
\begin{itemize}
\item[(i)] $\Gamma(A)$ is a Lie algebra.
\item[(ii)] For $X,Y\in \Gamma(A)$ and $f\in C^\infty(M)$, the Lie bracket on $\Gamma(A)$ satisfies the following Leibniz-type rule:
\begin{equation}
\nonumber
[X,fY]=f[X,Y]+(\rho(X)f)Y.
\end{equation}
\end{itemize}
\end{definition}
\begin{proposition}
Let $(A,\rho,M)$ be a Lie algebroid.  Then
\begin{itemize}
\item[(i)] $\rho: \Gamma(A)\rightarrow \Gamma(TM)$ is a Lie algebra map, where the Lie bracket on $\Gamma(TM)$ is just the usual Lie bracket of vector fields.
\item[(ii)] $[fX,Y]=f[X,Y]-(\rho(Y)f)X$ for all $X,Y\in \Gamma(A)$, $f\in C^\infty(M)$.
\end{itemize}
\end{proposition}

\begin{proof}
(i): See Lemma 8.1.4 of \cite{DZ}.

(ii): Direct calculation.
\end{proof}

\begin{definition}
\label{LieAlgebroidMorph}
Let $(A,\rho,M)$ and $(A',\rho',M)$ be Lie algebroids over the same base space $M$.  A Lie algebroid homomorphism from $(A,\rho,M)$ to $(A',\rho',M)$ is a vector bundle map $\varphi: A\rightarrow A'$ such that 
\begin{itemize}
\item[(i)] $\varphi: \Gamma(A)\rightarrow \Gamma(A')$ is a Lie algebra map,
\item[(ii)] $\rho'\circ \varphi=\rho$.
\end{itemize}
\end{definition}

Every Lie groupoid $\mathcal{G}\rightrightarrows M$ has an associated Lie algebroid $(A\mathcal{G},\rho,M)$ which arises by considering the Lie algebra of left-invariant vector fields on $\mathcal{G}$. 
\begin{definition}
\label{LeftInvVecField}
Let $\mathcal{G}\rightrightarrows M$ be a Lie groupoid.  A vector field $\widetilde{X}$ on $\mathcal{G}$ is \textit{left-invariant} if 
\begin{equation}
\nonumber
(l_g)_{\ast,h} \widetilde{X}_h=\widetilde{X}_{gh}
\end{equation}
for all $g,h\in \mathcal{G}$, where $s(g)=t(h)$ and 
\begin{equation}
\nonumber
l_g: t^{-1}(s(g))\longrightarrow t^{-1}(t(g))
\end{equation}
is left multiplication by $g$.
\end{definition}
\noindent Let 
\begin{equation}
\nonumber
T^t\mathcal{G}:=\ker t_\ast\subset T\mathcal{G}.
\end{equation}
Since $t:\mathcal{G}\longrightarrow M$ is a surjective submersion, it follows that $T^t\mathcal{G}$ is a smooth sub-bundle of $T\mathcal{G}$.  In addition, define
\begin{equation}
\nonumber
A\mathcal{G}:=T^t\mathcal{G}|_{M},
\end{equation}
where we recall that $M$ is identified with the embedded submanifold of $\mathcal{G}$ consisting of the unit elements.  Let $\frak{X}_l(\mathcal{G})$ denote the left-invariant vector fields of $\mathcal{G}$.  It can be shown that $\frak{X}_l(\mathcal{G})$ is closed under the ordinary Lie bracket of vector fields on $\mathcal{G}$.  Consequently, $\frak{X}_l(\mathcal{G})$ is a Lie algebra itself.  Definition \ref{LeftInvVecField} implies that the map
\begin{equation}
\nonumber
\frak{X}_l(\mathcal{G})\longrightarrow \Gamma(A\mathcal{G}),\hspace*{0.1in} \widetilde{X}\mapsto \widetilde{X}|_M
\end{equation}
is a vector space isomorphism.   The inverse map sends a section $X\in \Gamma(A\mathcal{G})$ to the left-invariant vector field $\widetilde{X}$ on $\mathcal{G}$ defined by
\begin{equation}
\nonumber
(l_g)_{\ast,s(g)} X_{s(g)}=\widetilde{X}_{g}.
\end{equation}

\begin{theorem}
\label{LieAlgebroidOfLieGroupoid}
Let $\mathcal{G}\rightrightarrows M$ be a Lie groupoid.  For $X,Y\in \Gamma(A\mathcal{G})$, define
\begin{equation}
\nonumber
[X,Y]:=[\widetilde{X},\widetilde{Y}]|_M,
\end{equation}
where $\widetilde{X}$, $\widetilde{Y}$ are the left-invariant vector fields associated to $X$ and $Y$ respectively.  Also, let
\begin{equation}
\nonumber
\rho:=s_{\ast}|_{A\mathcal{G}},
\end{equation}
where  $s:\mathcal{G}\longrightarrow M$ is the source map.  Then $(A\mathcal{G},\rho,M)$ is a Lie algebroid.
\end{theorem}

\begin{proposition}
\label{InducedLieAlgebroidMorph}
Let $\mathcal{G}\rightrightarrows M$ and $\mathcal{H}\rightrightarrows M$ be Lie groupods and let $\varphi: \mathcal{G}\rightarrow \mathcal{H}$ be a Lie groupoid homomorphism, where the morphism on the base space is $id_M$.  Let $\widehat{\varphi}:=\varphi_\ast|_{A\mathcal{G}}$.  Then $\widehat{\varphi}: A\mathcal{G}\rightarrow A\mathcal{H}$ is a Lie algebroid homomorphism.  
\end{proposition}

\begin{example}
Any Lie algebra $\frak{g}$ is naturally a Lie algebroid over a point.  Specifically, the Lie algebra structure on $\Gamma(\frak{g})$ is induced by that of $\frak{g}$ under the natural vector space isomorphism $\Gamma(\frak{g})\simeq \frak{g}$, and the anchor map of $\frak{g}$ is (necessarily) the zero map.
\end{example}

\begin{example}
The tangent bundle $TM$ of a manifold $M$ is naturally a Lie algebroid where the Lie bracket on $\Gamma(TM)$ is just the usual Lie bracket of vector fields on $M$, and the anchor map is just the identity map $\rho:=id_{TM}$.  $(TM,id_{TM},M)$ is called the \textit{tangent algebroid}.

A direct calculation shows that the tangent algebroid is the associated Lie algebroid of the pair groupoid $M\times M\rightrightarrows M$.
\end{example}

\begin{example}
Let 
\begin{equation}
\nonumber
\psi: \frak{g}\longrightarrow \Gamma(TM),\hspace*{0.1in} x\mapsto x_M:=\psi(x)\in \Gamma(TM)
\end{equation}
be an action of a Lie algebra $\frak{g}$ on a manifold $M$, that is, $\psi$ is a Lie algebra homomorphism.  Consider the trivial vector bundle 
\begin{equation}
\nonumber
\frak{g}\times M\rightarrow M.
\end{equation}
The sections of $\frak{g}\times M$ are naturally identified with smooth $\frak{g}$-valued functions on $M$.  Given two smooth functions $\phi, \tau: M\longrightarrow \frak{g}$, define
\begin{equation}
\label{ActionAlgebroidBracket}
[\phi,\tau](p):=[\phi(p),\tau(p)]+(\phi(p)_M)_p\tau - (\tau(p)_M)_p\phi
\end{equation}
for all $p\in M$, where $[\phi(p),\tau(p)]$ is understood to be the Lie bracket of $\phi(p),\tau(p)\in \frak{g}$ on $\frak{g}$.  Also, define 
\begin{equation}
\label{ActionAlgebroidAnchor}
\rho: \frak{g}\times M\longrightarrow TM,\hspace*{0.1in} (x,p)\mapsto (x_M)_p\in T_pM.
\end{equation}
Then $\frak{g}\times M$ is a Lie algebroid with bracket given by (\ref{ActionAlgebroidBracket}) and anchor map given by (\ref{ActionAlgebroidAnchor}).  $(\frak{g}\times M,\rho,M)$ is called the \textit{action algebroid}.

Now let $G$ be a Lie group whose Lie algebra is $\frak{g}$ and suppose that $M$ has a smooth left-action by $G$.  The $G$-action on $M$ induces an action of $\frak{g}$ on $M$ which sends $x\in \frak{g}$ to the vector field $x_M$ on $M$ given by
\begin{equation}
\label{Induced-g-action}
(x_M)_p:=\frac{d}{dt}|_{t=0}\mbox{exp}(-tx)p\in T_pM.
\end{equation}
The action algebroid given by the $\frak{g}$-action of (\ref{Induced-g-action}) coincides with the associated Lie algebroid of the action groupoid $G\times M\rightrightarrows M$.
\end{example}

\subsection{The exterior derivative of a Lie algebroid}

Every Lie algebroid $(A,\rho,M)$ has an \textit{exterior derivative} 
\begin{equation}
\nonumber
d_A: \Gamma(\wedge^k A^\ast)\longrightarrow \Gamma(\wedge^{k+1} A^\ast),
\end{equation}
which is analogous to the usual exterior derivative of differential forms.  Formally, $d_A$ is defined by
\begin{align}
\nonumber
(d_A\omega)(X_1,\dots, X_{k+1})=&\sum_{i=1}^{k+1}(-1)^{i+1}\rho(X_i)[\omega(X_1,\dots, \widehat{X}_i,\dots,X_{k+1})]\\
\label{dADef}
&+\sum_{i<j}(-1)^{i+j}\omega([X_i,X_j],X_1,\dots, \widehat{X}_i,\dots,\widehat{X}_j,\dots, X_{k+1})
\end{align}
for $\omega\in \Gamma(\wedge^k A^\ast)$,  $X_i\in \Gamma(A)$, $i=1,\dots, k+1$, where $ \widehat{X}_i$ denotes omission of $X_i$.  A direct calculation shows that 
\begin{equation}
d_A^2=0.
\end{equation}

\begin{example}
The exterior derivative $d_{TM}$ associated to the tangent algebroid $(TM,id_{TM}, M)$ is just the ordinary exterior derivative of differential forms on $M$.
\end{example}

\begin{example}
For a Lie algebra $\frak{g}$, the exterior derivative $d_{\frak{g}}$ associated to its natural Lie algebroid structure is given explicitly by
\begin{align}
\nonumber
(d_{\frak{g}}\omega)&(x_1,\dots,x_{k+1})\\
\nonumber
&=\sum_{i<j}(-1)^{i+j}\omega([x_i,x_j],x_1,\dots, \widehat{x}_i,\dots,\widehat{x}_j,\dots, x_{k+1}),
\end{align} 
for $\omega\in \wedge^k \frak{g}^\ast$, $x_1,\dots, x_{k+1}\in \frak{g}$.  From this, one sees that $d_{\frak{g}}$ is just the coboundary map in the Lie algebra cohomology of $\frak{g}$ with values in $\mathbb{R}$, where $\frak{g}$ acts trivially on $\mathbb{R}$.  
\end{example}

\subsection{quasi-Frobenius Lie algebroids}
As mentioned previously, a symplectic Lie algebroid over a point is a symplectic Lie algebra (or quasi-Frobenius Lie algebra as it is also called).   However, the name \textit{symplectic Lie algebra} also has a different meaning.  It also refers to $\frak{sp}(2n,\mathbb{R})$, the Lie algebra of the Lie group of $2n\times 2n$ symplectic matrices.  For this reason, we prefer to use the name \textit{quasi-Frobenius Lie algebroids} in place of symplectic Lie algebroids.  Formally, a quasi-Frobenius Lie algebroid is defined as follows: 
\begin{definition}
A \textit{quasi-Frobenius Lie algebroid} is a Lie algebroid $(A,\rho,M)$ together with a non-degenerate 2-cocycle $\omega$ in the Lie algebroid cohomology of $(A,\rho,M)$, that is, $\omega\in \Gamma(\wedge^2A^\ast)$ such that 
\begin{itemize}
\item[(i)] $\omega$ is nondegenerate
\item[(ii)] $d_A \omega = 0$.
\end{itemize}
Furthermore, if there exists $\theta\in \Gamma(A^\ast)$ such that $\omega=d_A\theta$, then $(A,\rho,M,\theta)$ is called a \textit{Frobenius Lie algebroid}. 
\end{definition}

\begin{example}
Let  $(\frak{q},\beta)$ be a quasi-Frobenius Lie algebra, that is, $\frak{q}$ is a Lie algebra and $\beta\in \wedge^2\frak{q}^\ast$ is a nondegenerate 2-cocycle in the Lie algebra cohomology of $\frak{q}$ with values in $\mathbb{R}$ (where $\frak{q}$ acts trivially on $\mathbb{R}$).  Let $d_{\frak{q}}$ denote the exterior derivative from the natural Lie algebroid structure on $\frak{q}$.  As noted previously, $d_{\frak{q}}$ coincides with the coboundary map in the Lie algebra cohomology of $\frak{q}$ with values in $\mathbb{R}$.  Hence, $d_{\frak{q}}\beta=0$.  Equipping $\frak{q}$ with its natural Lie algebroid structure, it follows that $(\frak{q},\beta)$ is naturally a quasi-Frobenius Lie algebroid over a point.
\end{example}

\begin{example}
Let $(M,\omega)$ be a symplectic manifold and let 
\begin{equation}
\nonumber
(TM,id_{TM},M)
\end{equation}
denote the tangent algebroid.  As noted previously, $d_{TM}=d$ where $d$ is the usual exterior derivative of differential forms on $M$.  From this, it follows that $(TM,id_{TM},M)$ together with $\omega$ is a quasi-Frobenius Lie algebroid over $M$.
\end{example}

\section{$t$-symplectic Lie groupoids}
In this section, we identify the Lie groupoid analogue of a symplectic Lie group.  To start, recall that a symplectic Lie group is a Lie group $G$ together with a left-invariant symplectic form $\omega$.  The condition of left-invariance simply means that 
\begin{equation}
\label{LeftInvarianceG}
l_g^\ast \omega=\omega,\hspace*{0.1in}\forall ~g\in G,
\end{equation}  
where $l_g: G\longrightarrow G$ is left translation by $g\in G$.  For a Lie groupoid $\mathcal{G}\rightrightarrows M$, where $M$ consists of more than one point, the condition of left-invariance given by equation (\ref{LeftInvarianceG}) is no longer applicable.  In other words, while the notion of left-invariant vector fields extends from Lie groups to Lie groupoids, the notion of left-invariant differential forms does not.   This is a consequence of the fact that multiplication on a groupoid $\mathcal{G}\rightrightarrows M$ is only partial whenever $M$ consists of more than one point.  

For $g\in \mathcal{G}$, the domain of $l_g$ is not $\mathcal{G}$.  Instead, one has 
\begin{equation}
\nonumber
l_g: t^{-1}(s(g))\stackrel{\sim}{\longrightarrow} t^{-1}(t(g))\hookrightarrow \mathcal{G},
\end{equation}
where $s$ and $t$ denote the source and target maps on $\mathcal{G}\rightrightarrows M$.  Consequently, if one starts with a differential form $\omega$ on $\mathcal{G}$, then the pullback $(l_g)^\ast\omega$ is now a differential form on the embedded submanifold $t^{-1}(s(g))$, rather than on $\mathcal{G}$.  

In the case of a Lie group $G$, every left-invariant $k$-form on $G$ is uniquely determined by some element in $\wedge^k \frak{g}^\ast$, where $\frak{g}$ is the Lie algebra of $G$.  Since $A\mathcal{G}$ is the analogue of $\frak{g}$ for a Lie groupoid $\mathcal{G}\rightrightarrows M$ and $\Gamma(\wedge^k\frak{g}^\ast)\simeq \wedge^k\frak{g}^\ast$, it is natural to take the Lie groupoid analogue of left-invariant $k$-forms on $\mathcal{G}\rightrightarrows M$ to be in one to one correspondence with the elements of $\Gamma(\wedge^k (A\mathcal{G})^\ast)$.  This motivates the following definition:
\begin{definition}
\label{LeftInvariance}
Let $\mathcal{G}\rightrightarrows M$ be a Lie groupoid.  A section 
\begin{equation}
\nonumber
\widetilde{\omega}\in \Gamma(\wedge^k(T^t\mathcal{G})^\ast)
\end{equation}
is \textit{left-invariant} if 
\begin{equation}
\nonumber
(l_g)^\ast(\widetilde{\omega}|_{t^{-1}(t(g))})=\widetilde{\omega}|_{t^{-1}(s(g))}
\end{equation}
for all $g\in \mathcal{G}$.
\end{definition}
\begin{remark}
\label{omegaP}
Recall that for $p\in M$ and $g\in t^{-1}(p)$,
\begin{equation}
\label{omegaPe1}
(T^t\mathcal{G})_g:=\ker~t_{\ast,g}=T_g~t^{-1}(p).
\end{equation}
This implies that $\widetilde{\omega}|_{t^{-1}(p)}$ in Definition \ref{LeftInvariance} is indeed a differential $k$-form on $t^{-1}(p)$.  
\end{remark}
\begin{remark}
Note that when $M$ is a point, that is, $\mathcal{G}$ is a Lie group, we have $T^t\mathcal{G}=T\mathcal{G}$, and Definition \ref{LeftInvariance} coincides with the usual notion of left-invariant differential forms on a Lie group.
\end{remark}
\noindent The next result justifies Definition \ref{LeftInvariance}. 
\begin{proposition}
\label{LeftInvCorr}
Let $\mathcal{G}\rightrightarrows M$ be a Lie groupoid and let $\Gamma_L(\wedge^k(T^t\mathcal{G})^\ast)$ denote the space of left-invariant sections of $\wedge^k(T^t\mathcal{G})^\ast$.  Define
\begin{equation}
\nonumber
\varphi: \Gamma_L(\wedge^k(T^t\mathcal{G})^\ast)\longrightarrow \Gamma(\wedge^k (A\mathcal{G})^\ast)
\end{equation}
by $\varphi(\widetilde{\omega}):=\widetilde{\omega}|_M$, where $p\in M$ is identified with its corresponding unit element $e_p\in \mathcal{G}$ and $A\mathcal{G}$ is the Lie algebroid of $\mathcal{G}\rightrightarrows M$.   Let $\widetilde{\omega}^{(p)}:=\widetilde{\omega}|_{t^{-1}(p)}$ and $\omega:=\varphi(\widetilde{\omega})$.  Then $\varphi$ is a vector space isomorphism and
\begin{equation}
\label{LeftInvCorre1}
[(d\widetilde{\omega}^{(p)})(\widetilde{X}_1,\dots, \widetilde{X}_{k+1})](g)=[(d_{A\mathcal{G}}\omega)(X_1,\dots, X_{k+1})](s(g)),
\end{equation}
for all $\widetilde{\omega}\in \Gamma_L(\wedge^k(T^t\mathcal{G})^\ast)$, $p\in M$, $g\in t^{-1}(p)$, and $X_i \in \Gamma(A\mathcal{G})$ for $i=1,\dots, k+1$, where $\widetilde{X}_i\in \Gamma(T^t\mathcal{G})$ is the left-invariant vector field associated to $X_i$.   
\end{proposition}
\begin{proof}
The linearity of $\varphi$ is clear.  We now show that $\varphi$ is injective.  Let $X_i \in \Gamma(A\mathcal{G})$, $i=1,\dots, k$ be arbitrary sections of $A\mathcal{G}$ and let $\widetilde{X}_i$, $i=1,\dots, k$ denote the corresponding left-invariant vector fields on $\mathcal{G}$.  Let $p\in M$ and $g\in t^{-1}(p)$.  Note that by equation (\ref{omegaPe1}), the restriction of $\widetilde{X}_i$ to $t^{-1}(p)$ is a vector field on $t^{-1}(p)$.  Let $\widetilde{\omega}\in \Gamma_L(\wedge^k(T^t\mathcal{G})^\ast)$ and $\omega:=\widetilde{\omega}|_M$.  Using the left-invariance of $\widetilde{\omega}$, we have
\begin{align}
\nonumber
[\widetilde{\omega}^{(p)}(\widetilde{X}_1,\dots, \widetilde{X}_k)](g)&=\widetilde{\omega}^{(p)}_g((\widetilde{X}_1)_g,\dots, (\widetilde{X}_k)_g)\\
\nonumber
&=\widetilde{\omega}^{(p)}_g((l_g)_{\ast,s(g)}({X}_1)_{s(g)},\dots,(l_g)_{\ast,s(g)}({X}_k)_{s(g)})\\
\nonumber
&=(l_g^\ast\widetilde{\omega}^{(p)})_{s(g)}(({X}_1)_{s(g)},\dots,({X}_k)_{s(g)})\\
\nonumber
&=\widetilde{\omega}^{(s(g))}_{s(g)}(({X}_1)_{s(g)},\dots,({X}_k)_{s(g)})\\
\nonumber
&={\omega}_{s(g)}(({X}_1)_{s(g)},\dots,({X}_k)_{s(g)})\\
\label{LeftInvCorre2}
&=[\omega(X_1,\dots, X_k)](s(g)).
\end{align}
Equation (\ref{LeftInvCorre2}) can be rewritten more generally as 
\begin{equation}
\label{LeftInvCorre3}
\widetilde{\omega}(\widetilde{X}_1,\dots, \widetilde{X}_k)=s^\ast[\omega(X_1,\dots, X_k)],
\end{equation}
where we recall that $\widetilde{\omega}^{(p)}$ is just the restriction of $\widetilde{\omega}$ to $t^{-1}(p)$.   Equation (\ref{LeftInvCorre3}) implies that  $\widetilde{\omega}$ is uniquely determined by $\omega:=\widetilde{\omega}|_M\in \Gamma(\wedge^k(A\mathcal{G})^\ast)$.  Hence, $\varphi$ is injective.  On the other hand, if $\beta\in  \Gamma(\wedge^k (A\mathcal{G})^\ast)$, then one obtains an element $\widetilde{\beta}\in \Gamma_L(\wedge^k(T^t\mathcal{G})^\ast)$ by defining
\begin{equation}
\label{LeftInvCorre4}
\widetilde{\beta}_g(u_1,\dots, u_k):=\beta_{s(g)}((l_{g^{-1}})_{\ast,g}u_1,\dots, (l_{g^{-1}})_{\ast,g}u_k).
\end{equation}
for $g\in \mathcal{G}$, $u_1,\dots, u_k\in (T^t\mathcal{G})_g$.   From the definition, it follows that $\widetilde{\beta}|_M=\beta$.  Hence, $\varphi$ is also surjective which proves that $\varphi$ is an isomorphism.  

Next, let $X_{k+1}\in \Gamma(A\mathcal{G})$ and let $\widetilde{X}_{k+1}$ be the associated left-invariant vector field on $\mathcal{G}$.  Let $g\in \mathcal{G}$.  Then
\begin{align}
\nonumber
\left[\widetilde{X}_{k+1}(\widetilde{\omega}(\widetilde{X}_1,\dots, \widetilde{X}_k))\right](g)&=(\widetilde{X}_{k+1})_g(\widetilde{\omega}(\widetilde{X}_1,\dots, \widetilde{X}_k))\\
\nonumber
&=(\widetilde{X}_{k+1})_g \left(s^\ast[\omega(X_1,\dots, X_k)]\right)\\
\nonumber
&=((l_g)_{\ast,s(g)} (X_{k+1})_{s(g)})\left([\omega(X_1,\dots, X_k)]\circ s\right)\\
\nonumber
&= (X_{k+1})_{s(g)}\left([\omega(X_1,\dots, X_k)]\circ s\circ l_g\right)\\
\nonumber
&= (X_{k+1})_{s(g)}\left([\omega(X_1,\dots, X_k)]\circ  s\right)\\
\nonumber
&=s_{\ast,s(g)}((X_{k+1})_{s(g)})[\omega(X_1,\dots, X_k)]\\
\label{LeftInvCorre5}
&=\rho((X_{k+1})_{s(g)})[\omega(X_1,\dots, X_k)]
\end{align}
where the second equality follows from equation (\ref{LeftInvCorre3}), the fifth equality follows from the fact that $s\circ l_g=s|_{t^{-1}(s(g))}$, and the last equality follows from the fact that the anchor map associated to $A\mathcal{G}$ is $\rho=s_{\ast}|_{A\mathcal{G}}$.  Equation (\ref{LeftInvCorre5}) can be written more generally as
\begin{equation}
\label{LeftInvCorre6}
\widetilde{X}_{k+1}(\widetilde{\omega}(\widetilde{X}_1,\dots, \widetilde{X}_k))=s^{\ast}\left(\rho(X_{k+1})[\omega(X_1,\dots, X_k)] \right).
\end{equation}
Now let $p\in M$ and $g\in t^{-1}(p)$.  Then
\begin{align}
\nonumber
[(d\widetilde{\omega}^{(p)})(\widetilde{X}_{1},\dots, &\widetilde{X}_{k+1})](g)=\sum_{i=1}^{k+1}(-1)^{i+1}(\widetilde{X}_i)_g[\widetilde{\omega}^{(p)}(\widetilde{X}_1,\dots,\widehat{\widetilde{X}}_i,\dots, \widetilde{X}_{k+1})]\\
\nonumber
&+\sum_{i<j}(-1)^{i+j} [\widetilde{\omega}^{(p)}([\widetilde{X}_i,\widetilde{X}_j],\widetilde{X}_1,\dots, \widehat{\widetilde{X}}_i,\dots, \widehat{\widetilde{X}}_j,\dots, \widetilde{X}_{k+1})](g)\\
\nonumber
&=\sum_{i=1}^{k+1}(-1)^{i+1}(\widetilde{X}_i)_g[\widetilde{\omega}(\widetilde{X}_1,\dots,\widehat{\widetilde{X}}_i,\dots, \widetilde{X}_{k+1})]\\
\nonumber
&+\sum_{i<j}(-1)^{i+j} [\widetilde{\omega}([\widetilde{X}_i,\widetilde{X}_j],\widetilde{X}_1,\dots, \widehat{\widetilde{X}}_i,\dots, \widehat{\widetilde{X}}_j,\dots, \widetilde{X}_{k+1})](g)\\
\nonumber
&=\sum_{i=1}^{k+1}(-1)^{i+1}\rho((X_i)_{s(g)})[{\omega}({X}_1,\dots,\widehat{{X}}_i,\dots, {X}_{k+1})]\\
\nonumber
&+\sum_{i<j}(-1)^{i+j} [\omega([{X}_i,{X}_j],{X}_1,\dots, {\widehat{X}}_i,\dots,{\widehat{X}}_j,\dots, {X}_{k+1})](s(g))\\
\nonumber
&=[(d_{A\mathcal{G}}\omega)(X_1,\dots, X_{k+1})](s(g))
\end{align}
where the third equality follows from equations (\ref{LeftInvCorre3}) and (\ref{LeftInvCorre6}) and the fact that $[X_i,X_j]:=[\widetilde{X}_i,\widetilde{X}_j]|_M$.  This completes the proof.
\end{proof}

\noindent We now define the Lie groupoid analogue of a symplectic Lie group.  The motivation for this definition will become clear shortly.

\begin{definition}
\label{tSymplecticDef}
A \textit{$t$-symplectic Lie groupoid} is a Lie grouopid $\mathcal{G}\rightrightarrows M$ together with a left-invariant section $\widetilde{\omega}\in \Gamma(\wedge^2 (T^t\mathcal{G})^\ast)$ with the property that $(t^{-1}(p),\widetilde{\omega}|_{t^{-1}(p)})$ is a symplectic manifold for all $p\in M$.  The section $\widetilde{\omega}$ is called a \textit{$t$-symplectic form} on $\mathcal{G}\rightrightarrows M$.
\end{definition}
\noindent Here are some immediate consequences of Definition \ref{LeftInvariance} and Definition \ref{tSymplecticDef}:
\begin{corollary}
\label{tSymplecticCor}
Let $\mathcal{G}\rightrightarrows M$ be a Lie groupoid and let $\widetilde{\omega}\in \Gamma(\wedge^2(T^t\mathcal{G})^\ast)$.  Then $\widetilde{\omega}$ is a $t$-symplectic form iff $(t^{-1}(p),\widetilde{\omega}|_{t^{-1}(p)})$ is a symplectic manifold for all $p\in M$ and 
\begin{equation}
\nonumber
l_g: (t^{-1}(s(g)),\widetilde{\omega}|_{t^{-1}(s(g))})\stackrel{\sim}{\longrightarrow}  (t^{-1}(t(g)),\widetilde{\omega}|_{t^{-1}(t(g))})
\end{equation}   
is a symplectomorphism for all $g\in \mathcal{G}$.
\end{corollary}

\begin{corollary}
Let $(\mathcal{G}\rightrightarrows M,\widetilde{\omega})$ be a $t$-symplectic Lie groupoid.  Then $\dim \mathcal{G}-\dim M$ is even. 
\end{corollary}
\begin{proof}
Let $p\in M$.  By Definition \ref{tSymplecticDef}, $(t^{-1}(p),\widetilde{\omega}|_{t^{-1}(p)})$ is a symplectic manifold.  Hence $t^{-1}(p)$ is an even-dimensional manifold.  Since $t$ is a submersion, we have
\begin{equation}
\nonumber
\dim t^{-1}(p)=\dim \mathcal{G}-\dim M.
\end{equation}  
This completes the proof.
\end{proof}

\begin{theorem}
\label{qFLAThm1}
Let $\mathcal{G}\rightrightarrows M$ be a Lie groupoid.  There is a one-to-one correspondence between quasi-Frobenius Lie algebroid structures on $(A\mathcal{G},\rho,M)$ and $t$-symplectic Lie groupoid structures on $\mathcal{G}\rightrightarrows M$.  This correspondence is given as follows.  Let 
\begin{equation}
\nonumber
\varphi: \Gamma_L(\wedge^2(T^t\mathcal{G})^\ast)\stackrel{\sim}{\longrightarrow} \Gamma(\wedge^2(A\mathcal{G})^\ast),\hspace*{0.1in}\widetilde{\omega}\mapsto \widetilde{\omega}|_M
\end{equation}
be the vector space isomorphism of Proposition \ref{LeftInvCorr}.  Let $\widetilde{\omega}\in \Gamma_L(\wedge^2(T^t\mathcal{G})^\ast)$.  Then $(\mathcal{G}\rightrightarrows M,\widetilde{\omega})$ is a $t$-symplectic Lie groupoid iff $(A\mathcal{G},\rho,M,\varphi(\widetilde{\omega}))$ is a quasi-Frobenius Lie algebroid.
\end{theorem}

\begin{proof}
Suppose $(A\mathcal{G},\rho,M,\omega)$ is a quasi-Frobenius Lie algebroid.  By Proposition \ref{LeftInvCorr}, there exists a unique section $\widetilde{\omega}\in \Gamma_L(\wedge^2(T^t\mathcal{G})^\ast)$  such that $\widetilde{\omega}|_M=\omega$.   From the proof of Proposition \ref{LeftInvCorr}, $\widetilde{\omega}$ is given explicitly by
\begin{equation}
\label{qFLAThm1e1}
\widetilde{\omega}_g(u,v)=\omega_{s(g)}((l_{g^{-1}})_{\ast,g}u,(l_{g^{-1}})_{\ast,g}v)
\end{equation}
for $g\in \mathcal{G}$ and $u,v\in (T^t\mathcal{G})_g$.  Since $\omega$ is nondegenerate on $A\mathcal{G}$ and 
\begin{equation}
\nonumber
(l_g)_{\ast,h}: (T^t\mathcal{G})_h\stackrel{\sim}{\longrightarrow}  (T^t\mathcal{G})_{gh}
\end{equation}
is a vector space isomorphism for all $g\in \mathcal{G}$ and $h\in t^{-1}(s(g))$, it follows that $\widetilde{\omega}$ is nondegenerate on $T^t\mathcal{G}$.  In particular,  $\widetilde{\omega}^{(p)}:=\widetilde{\omega}|_{t^{-1}(p)}$ is nondegenerate on
\begin{equation}
\nonumber
Tt^{-1}(p)=(T^t\mathcal{G})|_{t^{-1}(p)}
\end{equation}
for all $p\in M$.  

Now let $X,Y\in \Gamma(A\mathcal{G})$ be arbitrary and let $\widetilde{X},\widetilde{Y}\in \Gamma(T^t\mathcal{G})$ be the associated left-invariant vector fields on $\mathcal{G}$.  By Proposition \ref{LeftInvCorr}, 
\begin{equation}
\label{qFLAThm1e2}
[(d\widetilde{\omega}^{(p)})(\widetilde{X},\widetilde{Y})](g)=[(d_{A\mathcal{G}}{\omega})({X},{Y})](s(g)),\hspace*{0.1in}\forall~p\in M,~g\in t^{-1}(p).
\end{equation}
Since $d_{A\mathcal{G}}\omega = 0$, equation (\ref{qFLAThm1e2}) implies that $d\widetilde{\omega}^{(p)}=0$ for all $p\in M$.  Hence, $\widetilde{\omega}^{(p)}$ is a closed and nondegenerate 2-form on $t^{-1}(p)$ for all $p\in M$.  Hence, $(t^{-1}(p),\widetilde{\omega}^{(p)})$ is a symplectic manifold for all $p\in M$.  

On the other hand, suppose that $(\mathcal{G}\rightrightarrows M,\widetilde{\omega})$ is a $t$-symplectic Lie groupoid for some $\widetilde{\omega}\in \Gamma_L(\wedge^2(T^t\mathcal{G})^\ast)$.  Let $\omega:=\widetilde{\omega}|_M$.   Since $(t^{-1}(p),\widetilde{\omega}|_{t^{-1}(p)})$ is a symplectic manifold and
\begin{equation}
\nonumber
(A\mathcal{G})_p=(T^t\mathcal{G})_{p}= T_{p} t^{-1}(p)
\end{equation}
for all $p\in M$ (where, as usual, we identify $M$ with the unit elements of $\mathcal{G}$),  it follows immediately that $\omega:=\widetilde{\omega}|_M$ is nongenerate on $A\mathcal{G}$.  Since $\widetilde{\omega}$ is left-invariant, Proposition \ref{LeftInvCorr} implies equation (\ref{qFLAThm1e2}).  Since $d\widetilde{\omega}^{(p)}=0$ for all $p\in M$ and $s: \mathcal{G}\longrightarrow M$ is surjective, it follows that $d_{A\mathcal{G}}\omega=0$ as well.  Hence, $(A\mathcal{G},\rho,M,\omega)$ is a quasi-Frobenius Lie algebroid.  

Since the above constructions are clearly inverse to one another, the one-to-one correspondence between $t$-symplectic Lie groupoid structures on $\mathcal{G}\rightrightarrows M$ and quasi-Frobenius Lie algebroid structures on $(A\mathcal{G},\rho,M)$ is established.
\end{proof}

\noindent We conclude this section with a few elementary examples of $t$-symplectic Lie groupoids.

\begin{example}
Every symplectic Lie group is naturally a $t$-symplectic Lie groupoid over a point (and vice versa).
\end{example}

\begin{example}
Let $(M,\omega)$ be a symplectic manifold and let $M\times M\rightrightarrows M$ be the pair groupoid with
\begin{equation}
\nonumber
s(p,q):=q,\hspace*{0.2in} t(p,q):=p
\end{equation}
for $p,q\in M$.  Let $j: T^t (M\times M)\hookrightarrow T(M\times M)$ be the inclusion map and let 
\begin{equation}
\nonumber
\widetilde{\omega}:=j^\ast (s^\ast\omega)\in \Gamma(\wedge^2(T^t(M\times M))^\ast)
\end{equation}
 where $j^\ast$ denotes the dual map.  Then $M\times M\rightrightarrows M$ together with $\widetilde{\omega}$ is a $t$-symplectic Lie groupoid.  The associated quasi-Frobenius Lie algebroid is just the tangent algebroid $(TM,id_{M},M)$ with $\omega$ as the nondegenerate 2-cocycle.
\end{example}

\begin{example}
\label{ActionGroupoidExample}
Let $M$ be a manifold with a smooth left-action by a Lie group $G$.   Let $G\times M\rightrightarrows M$ be the associated action groupoid.  

Now suppose that $G$ admits a left-invariant symplectic form $\omega$, that is, $(G,\omega)$ is a symplectic Lie group.  Then $\omega$ induces a $t$-symplectic form $\widetilde{\omega}$ on $G\times M\rightrightarrows M$.  To construct $\widetilde{\omega}$, let
\begin{equation}
\nonumber
j: T^t(G\times M)\hookrightarrow T(G\times M)
\end{equation}
be the inclusion map and let $\pi_1: G\times M\longrightarrow G$ denote the natural projection map.   Then $\widetilde{\omega}\in \Gamma(\wedge^2(T^t(G\times M))^\ast)$ is defined by $\widetilde{\omega}:= j^\ast(\pi_1^\ast\omega)$.  

We now verify that $\widetilde{\omega}$ satisfies the conditions of a $t$-symplectic form.  For $p\in M$, let
\begin{equation}
\nonumber
i_p: t^{-1}(p)\hookrightarrow G\times M
\end{equation}
be the inclusion.  Then 
\begin{equation}
\label{ActionGroupoidExample1}
\widetilde{\omega}|_{t^{-1}(p)}=i_p^\ast(\pi_1^\ast\omega)=(\pi_1\circ i_p)^\ast\omega.
\end{equation}
Equation (\ref{ActionGroupoidExample1}) together with the fact that $d\omega=0$ implies that 
\begin{equation}
\label{ActionGroupoidExample2}
d(\widetilde{\omega}|_{t^{-1}(p)})=0
\end{equation}
for all $p\in M$.  Now let $(g,p)\in t^{-1}(p)$ and let $u,v\in T_{(g,p)} t^{-1}(p)$.  Since $t^{-1}(p)=G\times \{p\}$, it follows that 
\begin{equation}
\nonumber
u=(x,0_p),\hspace*{0.1in} v=(y,0_p)
\end{equation}
for some $x,y\in T_g G$.  Hence,
\begin{equation}
\label{ActionGroupoidExample3}
(\widetilde{\omega}|_{t^{-1}(p)})_{(g,p)}(u,v)=((\pi_1\circ i_p)^\ast\omega)_{(g,p)}(u,v)=\omega_g(x,y).
\end{equation}
Since $\omega$ is nondegenerate, it follows that $\widetilde{\omega}|_{t^{-1}(p)}$ is also nondegenerate.  Hence, $t^{-1}(p)$  together with $\widetilde{\omega}|_{t^{-1}(p)}$ is a symplectic manifold.

All that remains to check is that for all $(g,p)\in G\times M$, the left-translation map
\begin{equation}
\label{ActionGroupoidExample4}
l_{(g,p)}: t^{-1}(g^{-1}p)=G\times \{g^{-1}p\}\longrightarrow t^{-1}(p)=G\times \{p\}
\end{equation}
is a symplectomorphism (where we note that $s(g,p)=g^{-1}p$ and $t(g,p)=p$).  This can be seen as follows:
\begin{align}
\nonumber
(l_{(g,p)})^\ast \widetilde{\omega}|_{t^{-1}(p)}&=(l_{(g,p)})^\ast ((\pi_1\circ i_p)^\ast\omega)\\
\nonumber
&=(\pi_1\circ i_p\circ l_{(g,p)})^\ast\omega\\
\nonumber
&=(l_g\circ \pi_1\circ i_{g^{-1}p})^\ast \omega\\
\nonumber
&=(\pi_1\circ i_{g^{-1}p})^\ast [(l_g)^\ast \omega]\\
\nonumber
&=(\pi_1\circ i_{g^{-1}p})^\ast\omega\\
\nonumber
&=\widetilde{\omega}|_{t^{-1}(g^{-1}p)},
\end{align}
where the fifth equality follows from the fact that $\omega$ is left-invariant.  This proves that $(G\times M\rightrightarrows M,\widetilde{\omega})$ is a $t$-symplectic Lie groupoid.

\end{example}

\begin{example}
\label{Example2}
Let $(M,\omega)$ be a symplectic manifold and let $(G,\beta)$ be a symplectic Lie group.  Let
\begin{equation}
\nonumber
M\times G\times M\rightrightarrows M
\end{equation}
be the Lie groupoid defined by
\begin{itemize}
\item[(i)] $s(q,g,p):=p$
\item[(ii)] $t(q,g,p):=q$
\item[(iii)] $(r,h,q)(q,g,p):=(r,hg,p)$
\item[(iv)] $u(p):=(p,e,p)$
\item[(v)] $i(q,g,p):=(p,g^{-1},q)$.
\end{itemize}
Let
\begin{equation}
\nonumber
\widehat{\omega}:=\pi_2^\ast\beta+\pi_3^\ast\omega,
\end{equation}
where $\pi_2: M\times G\times M\longrightarrow G$  and $\pi_3: M\times G\times M\longrightarrow M$ are the natural projection maps.  Also, define the following inclusion maps
\begin{equation}
\nonumber
j: T^t(M\times G\times M)\hookrightarrow T(M\times G\times M),\hspace*{0.2in}i_p: t^{-1}(p)\hookrightarrow M\times G\times M
\end{equation}
for $p\in M$.  Define $\widetilde{\omega}:=j^\ast\widehat{\omega}$.

Since $\beta$ and $\omega$ are symplectic forms on $G$ and $M$ respectively and
\begin{equation}
\nonumber
t^{-1}(p)=\{p\}\times G\times M,
\end{equation}
it follows immediately that $(t^{-1}(p),\widetilde{\omega}|_{t^{-1}(p)})$ is a symplectic manifold for all $p\in M$.  Furthermore, we have
\begin{align}
\nonumber
(l_{(q,g,p)})^\ast(\widetilde{\omega}|_{t^{-1}(q)})&=(l_{(q,g,p)})^\ast(i_q^\ast\widehat{\omega})\\
\nonumber
&=(i_q\circ l_{(q,g,p)})^\ast \widehat{\omega}\\
\nonumber
&=(i_q\circ l_{(q,g,p)})^\ast (\pi_2^\ast\beta+\pi_3^\ast\omega)\\
\nonumber
&=(\pi_2\circ i_q\circ l_{(q,g,p)})^\ast\beta+(\pi_3\circ i_q\circ l_{(q,g,p)})^\ast\omega\\
\nonumber
&=(l_g\circ \pi_2\circ i_p)^\ast\beta+(\pi_3\circ i_p)^\ast\omega\\
\nonumber
&=i_p^\ast\circ \pi_2^\ast\circ (l_g^\ast\beta)+i_p^\ast (\pi_3^\ast\omega)\\
\nonumber
&=i_p^\ast (\pi_2^\ast\beta+\pi_3^\ast\omega)\\
\nonumber
&=i_p^\ast \widehat{\omega}\\
\nonumber
&=\widetilde{\omega}|_{t^{-1}(p)},
\end{align} 
for all $p,q\in M$, $g\in G$, where the seventh equality follows from the left-invariance of $\beta$.  Hence, $M\times G\times M\rightrightarrows M$ together with $\widetilde{\omega}$ is a $t$-symplectic Lie groupoid.
\end{example}

\section{Symplectic Lie group bundles}
In this section, we introduce the notion of a \textit{symplectic Lie group bundle} (SLGB), which combines the notion of a $t$-symplectic Lie groupoid with that of a Lie group bundle\footnote{See \cite{McK1,McK2} for a review of Lie group bundles.}.  Formally, SLGBs are defined as follows:
\begin{definition}
\label{SLGBDef}
A \textit{symplectic Lie group bundle} consists of the following data: $(G,\omega,E,\pi,M,\widetilde{\omega})$, where 
\begin{itemize}
\item[(i)] $(G,\omega)$ is a symplectic Lie group
\item[(ii)] $\pi: E\rightarrow M$ is smooth fiber bundle with fiber $G$
\item[(iii)] $\widetilde{\omega}$ is a smooth section of $\wedge^2(\ker \pi_\ast)^\ast$
\end{itemize}
such that 
\begin{itemize}
\item[(a)] for all $p\in M$, the fiber $E_p:=\pi^{-1}(p)$ has a Lie group structure (where the smooth structure on the Lie group coincides with the smooth structure on $E_p$ as an embedded submanifold of $E$)
\item[(b)] for all $p\in M$, $\gamma_p^\ast\widetilde{\omega}$ is a left-invariant symplectic form on $E_p$, where 
\begin{equation}
\nonumber
\gamma_p: T(E_p)\hookrightarrow \ker \pi_\ast
\end{equation}
is the inclusion
\item[(c)] there exists a system of local trivializations 
\begin{equation}
\nonumber
\{\psi_i: \pi^{-1}(U_i)\stackrel{\sim}{\rightarrow} U_i\times G\}
\end{equation}
such that for all $i$ and $p\in U_i$,
\begin{equation}
\nonumber
\psi_{i,p}: (E_p,\gamma^\ast_p\widetilde{\omega})\stackrel{\sim}{\rightarrow} (G,\omega)
\end{equation}
is an isomorphism of symplectic Lie groups, where $\psi_{i,p}$ is the composition
\begin{equation}
\nonumber
E_p\stackrel{\psi_i}{\longrightarrow} \{p\}\times G\stackrel{\sim}{\longrightarrow} G
\end{equation}
\end{itemize}
\end{definition}

\begin{proposition}
\label{SLGBProp0}
Every SLGB has a canonical $t$-symplectic Lie groupoid structure. 
\end{proposition}

\begin{proof}
Let $(G,\omega,E,\pi,M,\widetilde{\omega})$ be a SLGB.  Since $E\rightarrow M$ is a Lie group bundle, one can also regard it as a Lie groupoid $E\rightrightarrows M$ as follows:
\begin{itemize}
\item[1.] the source and target maps are defined by $s=t:=\pi$
\item[2.] the unit map $u: M\rightarrow E$ is defined by $u(p):=1_p$ for all $p\in M$ where $1_p$ is the identity element on $E_p$
\item[3.] the groupoid multiplication is induced by the fiber-wise group multiplication: $E_p\times E_p\rightarrow E_p$
\item[4.] the inverse map $i: E\rightarrow E$ is induced by the fiber-wise inverse map $E_p\rightarrow E_p$
\end{itemize}
Since $s=t=\pi$, we have 
\begin{equation}
\nonumber
t^{-1}(s(x))=t^{-1}(t(x))=E_{\pi(x)},\hspace*{0.1in}\forall~x\in E.
\end{equation}
By Definition \ref{SLGBDef},
\begin{equation}
\nonumber
\gamma_{\pi(x)}^\ast \widetilde{\omega}=\widetilde{\omega}|_{E_{\pi(x)}}
\end{equation}
is a left-invariant symplectic form on $E_{\pi(x)}$ for all $x\in E$.  Hence,  
\begin{equation}
\nonumber
l_x: (E_{\pi(x)},\widetilde{\omega}|_{E_{\pi(x)}})\stackrel{\sim}{\rightarrow}(E_{\pi(x)},\widetilde{\omega}|_{E_{\pi(x)}})
\end{equation}
is a symplectomorphism.   By Corollary \ref{tSymplecticCor}, $\widetilde{\omega}$ is a $t$-symplectic form on $E\rightrightarrows M$.  This completes the proof.
\end{proof}

\begin{proposition}
\label{SLGBProp1}
Let $\mathcal{E}=(G,\omega,E,\pi,M,\widetilde{\omega})$ be a SLGB and let $(AE,\rho,M,\beta)$ be the associated quasi-Frobenius Lie algebroid (where $\mathcal{E}$ is equipped with its canonical $t$-symplectic Lie groupoid structure).  Then
\begin{itemize}
\item[(i)] $\rho\equiv 0$
\item[(ii)] the Lie bracket on $\Gamma(AE)$ is $C^\infty(M)$-bilinear; in particular, there is an induced Lie algebra structure on the fiber $(AE)_p$ for all $p\in M$
\item[(iii)] there exists a system of local trivializations
\begin{equation}
\nonumber
\{\varphi_i: \pi_{AE}^{-1}(U_i)\stackrel{\sim}{\rightarrow} U_i\times \frak{g}\}
\end{equation}
such that for all $i$ and $p\in U_i$, 
\begin{equation}
\nonumber
\varphi_{i,p}: ((AE)_p,\beta_p)\stackrel{\sim}{\rightarrow} (\frak{g},\omega_e)
\end{equation}
is an isomorphism of quasi-Frobenius Lie algebras, where $\pi_{AE}$ is the projection map from $AE$ to $M$, $(\frak{g},\omega_e)$ is the quasi-Frobenius Lie algebra associated to $(G,\omega)$, and $\varphi_{i,p}$ is the composition
\begin{equation}
\nonumber
(AE)_p\stackrel{\varphi_i}{\longrightarrow} \{p\}\times \frak{g}\stackrel{\sim}{\longrightarrow} \frak{g}.
\end{equation}
\end{itemize} 
\end{proposition}

\begin{proof}
(i) follows from the fact that 
\begin{equation}
\nonumber
AE:=(\ker t_{\ast})|_{u(M)}
\end{equation}
$\rho:=s_\ast|_{AE}$, and $s=t=\pi$ from the proof of Proposition \ref{SLGBProp0}.  (Recall that we regard $AE$ as a vector bundle over $M$ by identifying the unit element $1_p\in E$ with $p\in M$.)

(ii) follows from the Leibniz property of the Lie bracket on $\Gamma(AE)$ together with the fact that the anchor map $\rho$ is identically zero. 

For (iii), let 
\begin{equation}
\nonumber
\{\psi_i: \pi^{-1}(U_i)\stackrel{\sim}{\rightarrow} U_i\times G\}
\end{equation}
be a system of local trivialization on $\pi: E\rightarrow M$ such that for all $i$ and $p\in U_i$,
\begin{equation}
\nonumber
\psi_{i,p}: (E_p,\gamma^\ast_p\widetilde{\omega})\stackrel{\sim}{\rightarrow} (G,\omega)
\end{equation}
is an isomorphism of symplectic Lie groups, where $\gamma_p: T(E_p)\hookrightarrow \ker \pi_\ast$ is the inclusion.  For each $i$, the restriction of 
\begin{equation}
\nonumber
(\psi_i)_\ast: T(\pi^{-1}(U_i))\stackrel{\sim}{\rightarrow} T(U_i\times G)
\end{equation}
to $AE|_{U_i}$ induces a local trivialization
\begin{equation}
\nonumber
\varphi_i: AE|_{U_i}\stackrel{\sim}{\rightarrow} U_i\times \frak{g}.
\end{equation}
Furthermore, for all $i$ and $p\in U_i$,
\begin{equation}
\nonumber
\varphi_{i,p}: ((AE)_p,\beta_p)\stackrel{\sim}{\rightarrow} (\frak{g},\omega_e)
\end{equation}
is an isomorphism of quasi-Frobenius Lie algebras.  Indeed, this follows from the fact that $(AE)_p=T_{1_p}(E_p)$,
\begin{equation}
\nonumber
\beta_p=\widetilde{\omega}_{1_p}=[\gamma_p^\ast\widetilde{\omega}]_{1_p}
\end{equation}
by Theorem \ref{qFLAThm1}, and $\psi_{i,p}$ is an isomorphism of symplectic Lie groups. This completes the proof.

\end{proof}

\begin{remark}
The Lie algebroid appearing in Proposition \ref{SLGBProp1} is both a quasi-Frobenius Lie algebroid and a Lie algebra bundle\footnote{See \cite{McK1, McK2} for a review of Lie algebra bundles.}.  For this reason, it is only natural that we call a quasi-Frobenius Lie algebroid $(A,\rho,M,\beta)$ satisfying conditions (i)-(iii) of Proposition \ref{SLGBProp1} a \textit{quasi-Frobenius Lie algebra bundle} (QFLAB).
\end{remark}

We now give a characterization of general SLGBs.  To this end, we will make use of the following results: 

\begin{lemma}
\label{LeftInvarianceLemma}
Let $\varphi: G\rightarrow H$ be a Lie group homomorphism and let $\theta\in \Omega^k(H)$ be a left-invariant $k$-form.  Then $\varphi^\ast\theta\in \Omega^k(G)$ is also left-invariant.
\end{lemma}
\begin{proof}
This is just a direct calculation:
\begin{align}
\nonumber
l_g^\ast (\varphi^\ast\theta)&=(\varphi\circ l_g)^\ast\theta\\
\nonumber
&=(l_{\varphi(g)}\circ \varphi)^\ast\theta\\
\nonumber
&=\varphi^\ast (l_{\varphi(g)}^\ast\theta)\\
\nonumber
&=\varphi^\ast\theta.
\end{align}
This completes the proof.
\end{proof}

\begin{lemma}
\label{TrivialSLGB}
Let $(G,\omega)$ be a symplectic Lie group and let $M$ be a manifold.  Let $\pi_1: M\times G\rightarrow M$ and $\pi_2: M\times G\rightarrow G$ denote the natural projections.  Also, let $\tau: \ker (\pi_{1})_\ast\hookrightarrow T(M\times G)$ be the inclusion.  Then 
 \begin{equation}
\nonumber
(G,\omega,M\times G,\pi_1,M,\tau^\ast(\pi^\ast_2\omega))
\end{equation}
is a SLGB, where for all $p\in M$, the Lie group structure on $\pi_1^{-1}(p)$ is the natural one.
\end{lemma}

\begin{proof}
For $p\in M$, let 
\begin{equation}
\nonumber
\gamma_p: T(\{p\}\times G)\hookrightarrow \ker (\pi_1)_\ast
\end{equation}
and
\begin{equation}
\nonumber
\iota_p: \{p\}\times G\hookrightarrow M\times G
\end{equation}
denote the inclusion maps.  Note that 
\begin{align}
\nonumber
\gamma_p^\ast(\tau^\ast(\pi^\ast_2\omega))&=(\tau\circ \gamma_p)^\ast(\pi^\ast_2\omega)\\
\nonumber
&=(\iota_p)^\ast(\pi^\ast_2\omega)\\
\label{TrivialSLGBE1}
&=(\pi_2\circ \iota_p)^\ast\omega.
\end{align}
Then for $g\in G$, we have
\begin{align}
\nonumber
(l_{(p,g)})^\ast[\gamma_p^\ast(\tau^\ast(\pi^\ast_2\omega))]&=(l_{(p,g)})^\ast[(\pi_2\circ \iota_p)^\ast\omega]\\
\nonumber
&=(\pi_2\circ \iota_p)^\ast\omega\\
\nonumber
&=\gamma_p^\ast(\tau^\ast(\pi^\ast_2\omega))
\end{align}
where we have used (\ref{TrivialSLGBE1}) in the first and third equality and Lemma \ref{LeftInvarianceLemma} in the second equality, where we note that 
\begin{equation}
\nonumber
\pi_2\circ \iota_p: \{p\}\times G\rightarrow G
\end{equation}
is a Lie group isomorphism.  Hence, $\gamma_p^\ast(\tau^\ast(\pi^\ast_2\omega))$ is left-invariant.  Furthermore, (\ref{TrivialSLGBE1}) implies that $\omega$ is closed and non-degenerate, i.e., symplectic.  This proves that $(\pi_1^{-1}(p), \gamma_p^\ast(\tau^\ast(\pi^\ast_2\omega)))$ is a symplectic Lie group.  Lastly, the identity map
\begin{equation}
\nonumber
id: M\times G\rightarrow M\times G
\end{equation}
is the desired trivialization for the SLGB structure.  This completes the proof.  
\end{proof}

\begin{proposition}
\label{AutomorphismSLG}
Let $(G,\omega)$ be a connected symplectic Lie group and let $\mbox{Aut}(G,\omega)$ be the group of automorphisms of $(G,\omega)$.  Then $\mbox{Aut}(G,\omega)$ is a finite dimensional Lie group.   Furthermore, if $G$ is simply connected, then $\mbox{Aut}(G,\omega)\simeq \mbox{Aut}(\frak{g},\omega_e)$ as Lie groups, where $(\frak{g},\omega_e)$ is the quasi-Frobenius Lie algebra associated to $(G,\omega)$ and $ \mbox{Aut}(\frak{g},\omega_e)$ is the group of automorphisms of $(\frak{g},\omega_e)$.
\end{proposition}

\begin{proof}
Let $\frak{g}$ be the Lie algebra of $G$.  In \cite{Ch}, Chevalley proved that the automorphism group of any finite dimensional connected Lie group is again a finite dimensional Lie group.  To show that $\mbox{Aut}(G,\omega)$ is a Lie group, it suffices to show that $\mbox{Aut}(G,\omega)$ is a closed subset of the Lie group $\mbox{Aut}(G)$; the closed subgroup theorem \cite{Wa} then implies that $\mbox{Aut}(G,\omega)$ is an embedded Lie subgroup of $\mbox{Aut}(G)$.  

To this end, define
\begin{equation}
\nonumber
f:  \mbox{Aut}(G)\rightarrow \mbox{Aut}(\frak{g})\subset GL(\frak{g}),\hspace*{0.1in} \varphi\mapsto \varphi_\ast: \frak{g}\stackrel{\sim}{\rightarrow} \frak{g}.
\end{equation}
Since $G$ is connected, $\varphi\in \mbox{Aut(G)}$ is uniquely determined by $f(\varphi)\in \mbox{Aut}(\frak{g})$.  From \cite{Ch}, the Lie group structure on $\mbox{Aut(G)}$ is obtained by using $f$ to identify $\mbox{Aut(G)}$ with $\mbox{im}~f$, which is shown to be a closed subgroup of the Lie group $\mbox{Aut}(\frak{g})$ (which in turn is a closed subgroup of $GL(\frak{g})$).  

By definition,
\begin{equation}
\nonumber
\mbox{Aut}(G,\omega)=\{\varphi\in \mbox{Aut}(G)~|~\varphi^\ast\omega=\omega\}.
\end{equation}
Let $\varphi\in \mbox{Aut}(G)$ be a limit point of $\mbox{Aut}(G,\omega)$ and let $\{\varphi_n\}\subset \mbox{Aut}(G,\omega)$ be a sequence which converges to $\varphi$.  Since $f$ is a Lie group isomorphism from $\mbox{Aut}(G)$ to $\mbox{im}~f$ (in particular - a homeomorphism), we have
\begin{equation}
\nonumber
f(\varphi_n)=(\varphi_n)_\ast\rightarrow f(\varphi)=\varphi_\ast.
\end{equation}
Hence, for $x,y\in \frak{g}=T_eG$, we have
\begin{align}
\nonumber
\varphi^\ast\omega_e(x,y)&=\omega_e(\varphi_\ast x,\varphi_\ast y)\\
\nonumber
&=\lim_{n\rightarrow \infty} \omega_e((\varphi_n)_\ast x,(\varphi_n)_\ast y)\\
\nonumber
&=\lim_{n\rightarrow \infty}(\varphi_n)^\ast\omega_e(x,y)\\
\nonumber
&=\lim_{n\rightarrow \infty}[(\varphi_n)^\ast\omega]_e(x,y)\\
\nonumber
&=\lim_{n\rightarrow\infty}\omega_e(x,y)\\
\nonumber
&=\omega_e(x,y),
\end{align}
where we have used the fact that $\varphi_n\in \mbox{Aut}(G,\omega)$ in the second to last equality.  Hence, $\varphi^\ast\omega_e=\omega_e$.  Since $\omega$ is left-invariant and $\varphi\in \mbox{Aut}(G)$, it follows that $\varphi^\ast\omega=\omega$.  This shows that $\varphi\in \mbox{Aut}(G,\omega)$, which in turn implies that $\mbox{Aut}(G,\omega)$ is a closed subset of $\mbox{Aut}(G)$.

For the last part of Proposition \ref{AutomorphismSLG}, suppose that $G$ is simply connected.  Then
\begin{equation}
\nonumber
f: \mbox{Aut}(G)\stackrel{\sim}{\rightarrow} \mbox{Aut}(\frak{g})
\end{equation}
is a Lie group isomorphism.  In addition, note that $f(\mbox{Aut}(G,\omega))=\mbox{Aut}(\frak{g},\omega_e)$.  Hence, the restriction of $f$ to $\mbox{Aut}(G,\omega)$ gives a Lie group isomorphism from $\mbox{Aut}(G,\omega)$ to $\mbox{Aut}(\frak{g},\omega_e)$.   This completes the proof.   
\end{proof}

\noindent The following result provides an alternate way of viewing SLGBs: 

\begin{theorem}
\label{SLGBAlternate}
Let $(G,\omega)$ be a connected symplectic Lie group and let $\pi: E\rightarrow M$ be a smooth fiber bundle with fiber $G$.  Then $\pi: E\rightarrow M$ admits the structure of a SLGB if and only if there exists a system of local trivializations
\begin{equation}
\nonumber
\{\psi_i: \pi^{-1}(U_i)\rightarrow U_i\times G\}
\end{equation}
for which all the transition functions take their values in the Lie group $\mbox{Aut}(G,\omega)$.  
\end{theorem}

\begin{proof}
$(\Rightarrow)$.  Suppose $(G,\omega,E,\pi,M,\widetilde{\omega})$ is a SLGB.  By Definition \ref{SLGBDef}, there exists a system of local trivializations  
\begin{equation}
\nonumber
\{\psi_i: \pi^{-1}(U_i)\stackrel{\sim}{\rightarrow} U_i\times G\}
\end{equation}
such that 
\begin{equation}
\nonumber
\psi_{i,p}: (E_p,\gamma_p^\ast\widetilde{\omega})\stackrel{\sim}{\rightarrow}  (G,\omega)
\end{equation}
is an isomorphism of symplectic Lie groups for all $i$ and $p\in U_i$, where $\gamma_p: T(E_p)\hookrightarrow  \ker~(\pi_E)_\ast$ is the inclusion.  This implies that for all $i, ~j$ such that $U_i\cap U_j\neq \emptyset$ and for all $p\in U_i\cap U_j$, the map
\begin{equation}
\nonumber
\psi_{j,p}\circ \psi_{i,p}^{-1}: (G,\omega)\stackrel{\sim}{\rightarrow}(G,\omega)\hspace*{0.1in}g\mapsto \phi_{ji}(p)g
\end{equation}
is an automorphism of $(G,\omega)$, where $\phi_{ji}$ is the transition function associated to $\psi_j\circ \psi_i^{-1}$.  Hence, $\mbox{im}~\phi_{ji}\subset \mbox{Aut}(G,\omega)$.  

$(\Leftarrow)$.  On the other hand, suppose that $\pi: E\rightarrow M$ is a smooth fiber bundle with fiber $G$ for which there exists a system of local trivializations
\begin{equation}
\nonumber
\{\psi_i: \pi^{-1}(U_i)\stackrel{\sim}{\rightarrow} U_i\times G\}
\end{equation}
whose transition functions all take their values in the Lie group $\mbox{Aut}(G,\omega)$.  First, we define a Lie group structure on the fibers of $E$.  Let $p\in M$ and let $U_i$ be any open set such that $p\in U_i$.  The (abstract) group structure on the fiber $E_p$ is obtained by declaring
\begin{equation}
\nonumber
\psi_{i,p}: E_p\stackrel{\sim}{\rightarrow} G
\end{equation}
to be a group isomorphism.  Hence, for $g,h\in G$, the product, inverse, and identity on $E_p$ are given respectively by
\begin{equation}
\label{SLGBAlternateE1}
\psi_{i,p}^{-1}(g)\cdot \psi_{i,p}^{-1}(h):=\psi_{i,p}^{-1}(gh), \hspace*{0.1in} (\psi_{i,p}^{-1}(g))^{-1}:=\psi_{i,p}^{-1}(g^{-1})
\end{equation}
and $1_p:=\psi_{i,p}^{-1}(e)$, where $e$ is the identity element on $G$.  Since $\psi_{i,p}$ is a diffeomorphism, the above product and inverse maps are smooth with respect to the manifold structure on $E_p$.  Hence, (\ref{SLGBAlternateE1}) defines a Lie group structure on $E_p$. 

To show that the  group structure on $E_p$ is well-defined, let $U_j$ be another open set such that $p\in U_j$.   Let 
\begin{equation}
\nonumber
\phi_{ji}: U_i\cap U_j\rightarrow \mbox{Aut}(G,\omega)
\end{equation}
be the transition function associated to $\psi_j\circ \psi_i^{-1}$. Then 
\begin{equation}
\nonumber
\psi_{i,p}^{-1}(g)=\psi_{j,p}^{-1}(\phi_{ji}(p)g),\hspace*{0.1in}\forall~g\in G.
\end{equation}
Let $\cdot_i$ an $\cdot_j$ denote the group products defined by $\psi_{i,p}$ and $\psi_{j,p}$ respectively.  Since $\phi_{ji}(p)\in \mbox{Aut}(G,\omega)$, we have
\begin{align}
\nonumber
\psi_{i,p}^{-1}(g)\cdot_i \psi_{i,p}^{-1}(h)&=\psi_{i,p}^{-1}(gh)\\
\nonumber
&=\psi_{j,p}^{-1}(\phi_{ji}(p)(gh))\\
\nonumber
&=\psi_{j,p}^{-1}([\phi_{ji}(p)(g)][\phi_{ji}(p)(h)])\\
\nonumber
&=\psi_{j,p}^{-1}(\phi_{ji}(p)(g))\cdot_j\psi_{j,p}^{-1}(\phi_{ji}(p)(h))
\end{align}
for $g,h\in G$.  This implies that the group product on $E_p$ is well-defined.   In a similar fashion, one can show that the identity element and the inverse map on $E_p$ are also well-defined.

Next, for $p\in U_i$, define $\omega^{(p)}:=\psi_{i,p}^\ast \omega$.  Then $\omega^{(p)}$ is a symplectic form on $E_p$.  Moreover, since $\psi_{i,p}$ is a Lie group isomorphism, Lemma \ref{LeftInvarianceLemma} implies that $\omega^{(p)}$ is also left-invariant.  From the definition of $\omega^{(p)}$, it follows that 
\begin{equation}
\nonumber
\psi_{i,p}: (E_p,\omega^{(p)})\stackrel{\sim}{\rightarrow} (G,\omega)
\end{equation}
is an isomorphism of symplectic Lie groups.  To see that the definition of $\omega^{(p)}$ is well-defined, let $U_j$ be another open set such that $p\in U_j$.  Then 
\begin{align}
\nonumber
(\psi_{j,p}\circ \psi_{i,p}^{-1})^\ast\omega=(\phi_{ji}(p))^\ast\omega=\omega
\end{align}
since $\phi_{ji}(p)\in \mbox{Aut}(G,\omega)$.  This implies that 
\begin{equation}
\nonumber
\psi_{j,p}^\ast\omega=\psi_{i,p}^\ast\omega,
\end{equation}
which proves that $\omega^{(p)}$ is well-defined.

Lastly, we construct a section $\widetilde{\omega}$ of $\wedge^2(\ker \pi_\ast)^\ast$ such that 
\begin{equation}
\label{SLGBAlternateE2}
\gamma_p^\ast\widetilde{\omega}=\omega^{(p)},\hspace*{0.1in}\forall~p\in M
\end{equation}
where $\gamma_p: T(E_p)\hookrightarrow \ker \pi_\ast$ is the inclusion.  To begin, for each $i$, equip the bundle 
\begin{equation}
\nonumber
\pi_{1,i}:U_i\times G\rightarrow U_i
\end{equation}
with the SLGB structure given by Lemma \ref{TrivialSLGB}.   The $t$-symplectic form on $U_i\times G$ is then $\tau_i^\ast (\pi_{2,i}^\ast\omega)$, where 
\begin{equation}
\nonumber
\pi_{2,i}: U_i\times G\rightarrow G,\hspace*{0.1in} \tau_i: \ker~(\pi_{1,i})_\ast\hookrightarrow T(U_i\times G)
\end{equation}
are the natural maps.  Let $\widetilde{\pi}_{i}$ denote the restriction of $\pi$ to $\pi^{-1}(U_i)$ and define 
\begin{equation}
\nonumber
\widetilde{\psi}_i:=(\psi_i)_\ast|_{\ker (\widetilde{\pi}_{i})_\ast}.
\end{equation}
Since $\pi_{1,i}\circ \psi_i=\widetilde{\pi}_{i}$, it follows that
\begin{equation}
\nonumber
\widetilde{\psi}_i: \ker (\widetilde{\pi}_{i})_\ast\stackrel{\sim}{\rightarrow} \ker~(\pi_{1,i})_\ast
\end{equation}
is a vector bundle isomorphism.   

Now define $\widetilde{\omega}_i:=(\widetilde{\psi}_i)^\ast[\tau_i^\ast (\pi_{2,i}^\ast\omega)]$.  Let $p\in U_i$, $x\in E_p$, and
\begin{equation}
\nonumber
u,v\in \ker(\widetilde{\pi}_{i})_{\ast,x}= \ker \pi_{\ast,x}=T_x (E_p).
\end{equation}
Then
\begin{align}
\nonumber
(\widetilde{\omega}_i)_x(u,v)&=[(\widetilde{\psi}_i)^\ast[\tau_i^\ast (\pi_{2,i}^\ast\omega)]]_x(u,v)\\
\nonumber
&=[\tau_i^\ast (\pi_{2,i}^\ast\omega)]_{\psi_i(x)}(\widetilde{\psi}_i(u),\widetilde{\psi}_i(v))\\
\nonumber
&=[ (\pi_{2,i}^\ast\omega)]_{\psi_i(x)}(\tau_i(\widetilde{\psi}_i(u)),\tau_i(\widetilde{\psi}_i(v)))\\
\nonumber
&=[ (\pi_{2,i}^\ast\omega)]_{\psi_i(x)}(\widetilde{\psi}_i(u),\widetilde{\psi}_i(v))\\
\nonumber
&=\omega_{\pi_{2,i}\circ \psi_i(x)}((\pi_{2,i})_\ast\circ \widetilde{\psi}_i(u),(\pi_{2,i})_\ast\circ \widetilde{\psi}_i(v))\\
\nonumber
&=\omega_{\psi_{i,p}(x)}((\psi_{i,p})_\ast(u),(\psi_{i,p})_\ast(v))\\
\nonumber
&=((\psi_{i,p})^\ast\omega)_x(u,v)\\
\label{SLGBAlternateE3}
&=\omega^{(p)}_x(u,v).
\end{align}
This proves that for all pairs $i,j$ such that $U_i\cap U_j\neq \emptyset$, we have
\begin{equation}
\nonumber
\widetilde{\omega}_i=\widetilde{\omega}_j~\hspace*{0.2in}\mbox{on}\hspace*{0.2in} \pi^{-1}(U_i)\cap \pi^{-1}(U_j)=\pi^{-1}(U_i\cap U_j).
\end{equation}
Hence, the $\widetilde{\omega}_i$'s glue together to form a global section $\widetilde{\omega}\in \Gamma(\wedge^2(\ker~\pi)^\ast)$.  Moreover, since $\gamma_p: T(E_p)\hookrightarrow \ker~\pi_\ast$ is just the inclusion and $\widetilde{\omega}|_{\pi^{-1}(U_i)}=\widetilde{\omega}_i$ for all $i$, (\ref{SLGBAlternateE3}) implies $\gamma_p^\ast\widetilde{\omega}=\omega^{(p)}$.  This completes the proof.
\end{proof}

\noindent We conclude the paper with the following corollary which provides a simple recipe for generating SLGBs:

\begin{corollary}
Let $(G,\omega)$ be a connected symplectic Lie group and let $\pi: P\rightarrow M$ be any principal $\mbox{Aut}(G,\omega)$-bundle.  Then the associated fiber bundle
\begin{equation}
\nonumber
E:=(P\times G)/\mbox{Aut}(G,\omega)\rightarrow M
\end{equation}
admits the structure of a SLGB, where $\mbox{Aut}(G,\omega)$ acts naturally on $G$ from the left. 
\end{corollary}
\begin{proof}
From the definition of the associated fiber bundle, we see that $E\rightarrow M$ is a smooth fiber bundle with fiber $G$ which has a system of local trivializations whose transition functions all take their values in $\mbox{Aut}(G,\omega)$.  Theorem \ref{SLGBAlternate} now implies that $E\rightarrow M$ admits the structure of a SLGB.  
\end{proof}


\begin{thebibliography}{99}

\bibitem
{BC2013} O. Baues, V. Cort\'{e}s,
\textit{Symplectic Lie groups} I-III, arXiv:1307.1629 [math.DG]

\bibitem
{BottTu1982} R. Bott, L. Tu,
Differential Forms in Algebraic Topology, Springer 1982.

\bibitem
{CP1994} V. Chari, A. Pressley,
\textit{Quantum Groups}, 
Cambridge University Press, 1994.

\bibitem
{Ch} C. Chevalley, \textit{Theory of Lie Groups}, Princeton University Press, 1946.

\bibitem
{Chu1974} B. Chu,
\textit{Symplectic homogeneous spaces},
Trans. Amer. Math. Soc. vol. 197 (1974), 145-159.

\bibitem
{DZ} J. Dufour, N. Zung,
\textit{Poisson Structures and Their Normal Forms},
Berkh\"{a}user Verlag, 2005.


\bibitem
{KSM} Y. Kosmann-Schwarzbach, K. Mackenzie,
\textit{Differential operators and actions of Lie algebroids}, Quantization, Poisson brackets, and Beyond.,  Contemp. Math. vol. 315, pp. 213-233, Amer. Math. Soc., Providence, RI (2002).

\bibitem
{Lee} J. M. Lee,
\textit{Introduction to Smooth Manifolds},
Springer-Verlag, New York, 2003.

\bibitem
{LMM2005} M. de Leon, J. Marrero, E. Mart\'{i}nez,
\textit{Lagrangian submanifolds and dynamics on Lie algebroids},
J. Phys. A: Math. Gen., vol 38, no. 24, (2005), pp. 241-308.

\bibitem
{Marle} C.M. Marle,
\textit{Calculus on Lie algebroids, Lie groupoids and Poisson manifolds},
Dissertationes Mathematicae 457, Warszawa : Institute of Mathematics, Polish Academy of Sciences, 2008.

\bibitem
{Marle1} C.M. Marle,
\textit{Differential calculus on a Lie algebroid and Poisson manifolds},
arXiv:0804.2451v2 [math.DG], June 2008. 

\bibitem
{McK1} K. Macknezie, 
\textit{Lie Groupoids and Lie Algebroids in Differential Geometry}, London Mathematical Society Lecture Note Series, vol. 124,
Cambridge University Press, 1987.

\bibitem
{McK2} K. Mackenzie, 
\textit{General Theory of Lie Groupoids and Lie Algebroids}, London Mathematical Society Lecture Note Series, vol. 213,
Cambridge University Press, 2005.

\bibitem
{NT2001} R. Nest, B.  Tsygan,
\textit{Deformations of symplectic Lie algebroids, deformations of holomorphic symplectic structures, and index theorems}, Asian J. Math. \textit{5} (2001), 599-635.

\bibitem
{Wa} F. Warner, \textit{Foundations of differentiable manifolds and Lie groups}, Springer-Verlag, 1983. 

\bibitem
{Weinstein1987} A. Weinstein,
\textit{Symplectic groupoids and Poisson manifolds}
Bulletin of the Amer. Math. Soc., vol. 16, no.1 (1987), 101-104.





\end{thebibliography}
\end{document}